\documentclass[reqno,b5paper]{amsart}
\usepackage[utf8x]{inputenc}
\usepackage{ucs}
\usepackage{indentfirst}
\usepackage{enumerate}
\usepackage{hyperref}
\usepackage{url}
\newtheorem{teo}{\bf Theorem}
\newtheorem{cor}[teo]{\bf Corollary}
\newtheorem{lema}[teo]{\bf Lemma}

\theoremstyle{remark}

\theoremstyle{definition}
\newtheorem{defi}[teo]{\bf Definition}

\PrerenderUnicode{\unichar{537}}
\newcommand{\catname}[1]{{\normalfont\textbf{#1}}}
\newcommand{\FinSet}{\catname{FinSet}}
\newcommand{\Top}{\catname{Top}}
\newcommand{\Set}{\catname{Set}}

\newcommand{\BA}{\catname{BA}}
\newcommand{\ifff}{\Leftrightarrow}
\newcommand{\fp}{\mathfrak{p}}
\newcommand{\Fp}{\mathfrak{P}}
\newcommand{\Oo}{\mathcal{O}}

\title{Codensity and Stone spaces}
\author{Andrei Sipo\unichar{537}}
\address{Simion Stoilow Institute of Mathematics of the Romanian Academy, P.O. Box 1-764, 014700 Bucharest, Romania.}
\address{Faculty of Mathematics and Computer Science, University of Bucharest, Academiei 14, 010014 Bucharest, Romania.}
\email{Andrei.Sipos@imar.ro}

\begin{document}

\begin{abstract}
We present a detailed computation of two codensity monads associated to two canonical functors -- the inclusion functor of $\FinSet$ into $\Top$ and the inclusion functor of the category of the powers of the Sierpi\'nski space into $\Top$. We show that the categories of algebras of the two monads are the categories of Stone spaces and of sober spaces, respectively. A new motivation for defining these classes of spaces is therefore obtained.
\end{abstract}

\maketitle

\section{Introduction}

The codensity monad was first defined by Kock \cite{kock}, and afterwards applied to the ultrafilter case by Kennison and Gildenhuys \cite{kg}. Recently, there has been a resurgent interest in the concept, following Leinster's 2013 paper \cite{lein}. He pioneered an integration-oriented approach, which he successfully used to obtain a codensity perspective on double dual vector spaces and ultraproducts. (We shall not make use of that view here, sticking instead to Kock's original definition.) That approach is more amenable to measure-theoretic notions -- Leinster's suggestion of Giry's measure space monad, for example, was later picked up by Sturtz \cite{sturtz} and Avery \cite{avery}. Our goal will be to exhibit codensity monads whose categories of algebras are the category of Stone spaces and that of sober spaces, respectively.

Let $\mathcal{B}$ be an essentially small category (i.e., its skeleton is a small category) and $\mathcal{A}$ be a locally small category that has all small limits. If $G: \mathcal{B} \to \mathcal{A}$ is a functor, we have an induced functor
$$h_G : \mathcal{A} \to (\Set^{\mathcal{B}})^{op}$$
defined (on objects) by
$$h_G(A)(B) = \text{Hom}_{\mathcal{A}}(A,G(B))$$
where $A$ is an object of $\mathcal{A}$ and $B$ one of $\mathcal{B}$. Under our hypotheses, this functor has a right adjoint, denoted by $h'_G$ and defined on an arbitrary $Y$ in $\Set^{\mathcal{B}}$ as:
\begin{equation}\label{lim}
h'_G(Y) = \varprojlim_{B \in \mathcal{B},\ y \in Y(B)} G(B)
\end{equation}
and therefore the underlying endofunctor (of $\mathcal{A}$) of the associated monad\linebreak$(T^G, \eta^G, \mu^G)$ is given by:
$$T^G(A) = \varprojlim_{B \in \mathcal{B},\ f:A\to G(B)} G(B)$$

We now explain the limit notation. We do so for the general case in \eqref{lim}. Let $\mathcal{J}$ be the category whose objects are pairs $(B,y)$ where $B \in \mathcal{B}$ and $y \in Y(B)$, and so if $(B,y)$ and $(B',y')$ are two such objects, their Hom-set $\text{Hom}_{\mathcal{J}}((B,y),(B',y'))$ contains all morphisms $\phi : B \to B'$ in $\mathcal{B}$ such that $Y(\phi)(y) = y'$. We now define the functor $D:\mathcal{J} \to \mathcal{A}$ by $D((B,y))=G(B)$ and $D(\phi)=G(\phi)$. It is of this diagram $D$ that the limit is taken. (The limit exists because $\mathcal{B}$, and therefore also $\mathcal{J}$, is equivalent to a small category, and $\mathcal{A}$ has all small limits.)

In the particular case of the definition of $T^G$, the objects of $\mathcal{J}$ are pairs $(B,f)$ where $B \in \mathcal{B}$ and $\ f:A\to G(B)$, and so if $(B,f)$ and $(B',f')$ are two such objects, $\text{Hom}_{\mathcal{J}}((B,f),(B',f'))$ contains all morphisms $\phi : B \to B'$ in $\mathcal{B}$ such that $G(\phi) \circ f = f'$. In other words, $\mathcal{J}$ is the comma category given by the constant functor $!_A$ (from the terminal category) and our functor $G$. The functor $D$ is constructed similarly. (The projections will be denoted simply by $\pi_f$, as any $f$ contains the information necessary for recovering the corresponding $B$.)

The way that $T^G$ acts on morphisms is an immediate application of the limit's universal property: if $\phi: X \to Y$ is a morphism of $\mathcal{A}$, $T^G(\phi)$ is the unique arrow such that for every $(B,f) \in \mathcal{J}$, the following equation holds:
$$\pi_f \circ T^G(\phi) = \pi_{f\circ\phi}$$

We now turn to the unit and multiplication of the monad. The object $A$ with the canonical morphisms $f: A \to D((B,f))$ forms an obvious cone for $D$, and so we can define $\eta^G_A$ as the unique morphism such that for all $(B,f)$, we have $\pi_f \circ \eta^G_A = f$. Similarly, $\mu^G_A : T^G(T^G(A)) \to T^G(A)$ must satisfy the equations $\pi_f \circ \mu^G_A = \pi_{\pi_f}$ (the last projection arising from the second application of the functor $T^G$). The components of the natural transformations have therefore been given.

This monad is called the {\bf codensity monad} of $G$ (an immediate corollary is that, whenever $G$ has a left adjoint, the codensity monad of $G$ is precisely the monad corresponding to {\it that} adjunction). In the case where the adjoint $h'_G$ above does not exist (i.e., the assumptions about the categories $\mathcal{A}$ and $\mathcal{B}$ are weaker), we notice that the monad can still be defined using just the limit, if its existence can be ensured: the monad laws are immediate consequences of the equations above. (The monad in Section~\ref{sec:sober} will be defined in this way. There are also other options available, see \cite{leincat}.) If the unit of the monad is an isomorphism (or, equivalently, if $h_G$ is fully faithful), the functor $G$ is said to be {\bf codense}.

The result by Kennison and Gildenhuys \cite{kg}, referenced by Leinster \cite{lein}, says that the codensity monad of the inclusion functor $\FinSet \hookrightarrow \Set$ has as its underlying functor the {\it ultrafilter functor} $U$, defined on a set $A$ as:
$$U(A)=\text{the set of ultrafilters in the Boolean algebra $P(A)$},$$
where $P(A)$ is the powerset of $A$. 

In Section~\ref{sec:stone}, we see what happens when we replace $\Set$ by $\Top$, the category of all topological spaces -- namely, the image of the corresponding monad is the category of Stone spaces. In Section~\ref{sec:inter}, we motivate the definition of a sober space and the construction of another subcategory of $\Top$. In Section~\ref{sec:sober}, we prove the codensity link between the two.

A monad $(T,\eta,\mu)$ on a category $\mathcal{C}$ is called {\bf idempotent} if its counit $\mu$ is a natural isomorphism. In that case, it follows by the monad unit law that for any object $X$ of $\mathcal{C}$, we have that $\mu_X = (\eta_{T(X)})^{-1} = (T(\eta_X))^{-1}$. Also, by \cite[Proposition 4.2.3, (1)]{Bor2}, the forgetful functor $U: \mathcal{C}^T \to \mathcal{C}$ is full and faithful, therefore it induces an isomorphism between its essential image and $\mathcal{C}^T$. We show that this essential image equals the essential image of $T$. Firstly, let $(A, \alpha: T(A) \to A)$ be a $T$-algebra. By \cite[Proposition 4.2.3, (4)]{Bor2}, $A$ is isomorphic to $T(A)$ and hence it is in the essential image of $T$. Conversely, we take an object $X$ of $\mathcal{C}$ and we seek to prove that $(T(X), \mu_X : T^2(X) \to T(X))$ is a $T$-algebra. From the fact that $\mu_X = (\eta_{T(X)})^{-1}$ it follows that $id_{T(X)} = \mu_X \circ \eta_{T(X)}$, i.e. the first algebra law. In addition, $\mu_{T(X)} = (T(\eta_{T(X)}))^{-1}= T(\eta_{T(X)}^{-1}) = T(\mu_X)$, so $\mu_X \circ \mu_{T(X)} = \mu{X} \circ T(\mu_X)$, i.e. the second algebra law. We have shown that $\mathcal{C}^T$ is equivalent to the essential image of $T$. Since in the sequel we shall compute essential images of {\it idempotent} codensity monads, we shall also obtain characterizations of their categories of algebras.

We note that these results can be obtained in a shorter way, using previously obtained results. This alternate proof, pointed out by an anonymous reviewer, is sketched out in Section~\ref{sec:alter}. However, we feel that a detailed computation can shed some light on some subtle issues invoked, directly or indirectly, in both arguments.

\section{The case of Stone spaces}\label{sec:stone}

From now on, we will focus on the codensity monad of the inclusion functor $\FinSet \hookrightarrow \Top$, and as the class of objects of $\FinSet$'s skeleton is $\mathbb{N}$ and $\Top$ has all small limits, we are under the strict hypotheses considered so far.

Let us review some general theory. The clopen sets of a topological space have the structure of a Boolean algebra, and this gives a functor $C: \Top~\to~\BA^{op}$. In the other direction, the set of ultrafilters of a Boolean algebra $B$ can be enriched with a natural topology. A basis for that topology is, for example, composed of the sets $D_x = \{ U \mid x \in U \}$, for every $x \in B$. This space is compact, Hausdorff, and has a basis of clopen sets (since $D_{\neg x}$ is the complement of $D_x$), and therefore a {\bf Stone space}. We have obtained another functor $S: \BA^{op}~\to~\Top$. For every Boolean algebra $B$, $CS(B)$ is isomorphic to $B$ (this, and what follows, is part of {\it Stone duality}). Note that $SC$ is an endofunctor of $\Top$, which acts on morphisms $f: X \to Y$ as (for any ultrafilter $U$):
$$SC(f)(U)=\{A \in C(Y) \mid \{x \in X \mid f(x) \in A \} \in U \}.$$
Also, every space $X$ is equipped with a map $\eta_X : X \to SC(X)$, defined by:
$$\eta_X(x)=\{A \in C(X) \mid x \in A\}$$
and if $X$ is already a Stone space, $\eta_X$ is an isomorphism (all these maps are the components of a natural transformation $\eta$). All these basic facts can be found, for example, in the book of Bell and Slomson \cite{bellslom}, along with their proofs.

Another result that we shall need is the following lemma, a slight variant of one of Galvin and Horn \cite[p. 521]{gh}.

\begin{lema}\label{lgh}
Let $B$ be a Boolean algebra and $U \subseteq B$. Then the following are equivalent:
\begin{enumerate}[(i)]
\item $U$ is an ultrafilter;
\item if $x_0, x_1, x_2 \in B$ such that $x_0 \vee x_1 \vee x_2= 1$ and for every $0 \leq i < j < 3$, $x_i \wedge x_j = 0$, then there is exactly one $i$ such that $x_i \in U$.
\end{enumerate}
\end{lema}

\begin{proof}
We do only the hard direction, i.e. $(ii) \Rightarrow (i)$.

\begin{enumerate}[Step 1.]
\item By writing $1$ as $1 \vee 0 \vee 0$, it is clear that $1 \in U$ and $0 \notin U$.
\item Now, for $x \in B$, by writing $1$ as $x \vee \neg x \vee 0$, we see that exactly one of $x$ and $\neg x$ is in $U$.
\item Then, we show that if $x \in U$ and $x \leq y$ then $y \in U$. We write $1$ as $x \vee (y\wedge \neg x) \vee \neg y$, and since $x \in U$, $\neg y \notin U$, and so, by Step 2, $y \in U$.
\item Finally, we prove that if $x$ and $y$ are in $U$, $x \wedge y$ is in $U$. We firstly write $1$ as $(x \wedge \neg y) \vee (\neg x \wedge \neg y) \vee y$, and so, because $y \in U$, $x \wedge \neg y$ in not in $U$. Then, $1 = (x \ \wedge \neg y) \vee (x \wedge y) \vee \neg x$, and the first and the last are not in $U$, hence $x \wedge y \in U$.
\end{enumerate}
So $U$ is an ultrafilter.
\end{proof}

It is clear that this lemma holds for any finite set of at least three elements of $B$. (The easy, $(ii) \Rightarrow (i)$ direction, also holds for less than three elements.)

The main result of this section is the following:

\begin{teo}
The codensity monad of the inclusion functor $\FinSet \hookrightarrow \Top$ is $(SC,\eta,\mu)$, where, for every $X$, $\mu_X$ is defined as $(\eta_{SC(X)})^{-1}$.
\end{teo}

\begin{proof}
We begin by computing the endofunctor component, i.e. to show that:
$$SC(X)=\varprojlim_{B \in \FinSet,\ f:X\to B} B$$
Firstly, we define, for every such $f:X \to B$, the map $\psi_f : SC(X) \to B$ by:
$$\psi_f(U)=b \ifff f^{-1}(b)\in U,$$
for every and $U \in SC(X)$ and $b \in B$ (these maps will be the projections). Since $U$ is an ultrafilter, by the generalized variant of Lemma~\ref{lgh}, there is a single $b$ with this property and so the map is well-defined. It is also continuous, because for every $b$, we have:
\begin{align*}
\psi_f^{-1}(b)&=\{U \mid \psi_f(U)=b \} \\
&= \{ U \mid f^{-1}(b) \in U \} \\
&=D_{f^{-1}(b)}
\end{align*}
We claim that $(SC(X), \{\psi_f\}_{f:X \to B})$ is the required universal cone. First we show that it is a cone, i.e. that for every $f:X\to B$ as above and every map of finite sets $\phi: B \to B'$, the following equation holds:
\begin{equation}\label{eq:psi}
\phi \circ \psi_f = \psi_{\phi \circ f}
\end{equation}
For every $x \in X$ and $b \in B$, if $f(x)=b$ then $(\phi \circ f)(x)=\phi(b)$. So, for any $b$, $f^{-1}(b)\subseteq (\phi \circ f)^{-1}(\phi(b))$. We infer that if $f^{-1}(b)$ is in an ultrafilter $U$, also $(\phi \circ f)^{-1}(\phi(b))$ is in $U$; in other words, if $\psi_f(U)=b$ then $\psi_{\phi \circ f}(U)=\phi(b)$. Hence the equation.

Universality means that for an arbitrary cone $(M, \{\psi'_f\}_{f:X \to B})$ there exists a unique $\alpha : M \to SC(X)$ such that for every $f$, we have that:
$$\psi_f \circ \alpha = \psi'_f$$
We shall make use of {\it characteristic functions}. We denote by $\Omega$ the set $\{0,1\}$. If $X$ is a topological space and $A \in C(X)$, the characteristic function of $A$, $\chi_A : X \to \Omega$, is continuous.

The existence of $\alpha$ will be proven first. We define it as:
$$\alpha(m)=\{A \in C(X) \mid \psi'_{\chi_A}(m)=1 \}$$
for every $m \in M$. There must be shown: i) that it is well-defined, i.e. that $\alpha(m)$ is an ultrafilter for every $m$; ii) that it is continuous; and iii) that it commutes with the projections.

To prove (i), the following equation is useful:
\begin{equation}\label{eq:chi}
\chi_{f^{-1}(b)}=\chi_{\{b\}} \circ f,
\end{equation}
as it trivially holds for any $f$, $b$ as above. Now we can derive the following equivalences:
\begin{align*}
f^{-1}(b) \in \alpha(m) &\ifff \psi'_{\chi_{f^{-1}(b)}}(m)=1 \\
&\ifff \psi'_{\chi_{\{b\}} \circ f}(m)=1&(\text{by}~\ref{eq:chi}) \\
&\ifff \chi_{\{b\}}(\psi'_f(m)) = 1&(\text{by the analogue of}~\ref{eq:psi}) \\
&\ifff \psi'_f(m) = b.
\end{align*}

Let $\Omega'$ be the set $\{0,1,2\}$. To show that $\alpha(m)$ is an ultrafilter in $C(X)$, we firstly write $X$ as a partition of three (disjoint) clopen sets $X_0, X_1, X_2$. Let $f:X \to \Omega'$ be such that $f(x)=i$ if and only if $a \in X_i$. By the earlier equivalences, we have that:
\begin{align*}
X_i \in \alpha(m) &\ifff f^{-1}(i) \in \alpha(m) \\
&\ifff \psi'_f(m) = i,
\end{align*}
so there is exactly one $i$, namely $\psi'_f(m)$, such that $X_i \in \alpha(m)$. By Lemma~\ref{lgh}, $\alpha(m)$ is an ultrafilter for every $m$, i.e. (i).

Because of that, we can apply $\psi_f$ to $\alpha(m)$, and by using the definition of $\psi_f$ on the tautology $\psi_f(\alpha(m))=\psi_f(\alpha(m))$, we obtain $f^{-1}(\psi_f(\alpha(m))) \in \alpha(m)$, and, therefore, by using again the earlier equivalences, we have that $\psi'_f(m)=\psi_f(\alpha(m))$. This takes care of (iii).

What remains is continuity, i.e. (ii). It is sufficient to show that the inverse image of a basic clopen set is open. Let $A \in C(X)$. We have:
\begin{align*}
\alpha^{-1}(D_A)&=\{m \in M \mid \alpha(m) \in D_A\}\\
&=\{m \in M \mid A \in \alpha(m) \}\\
&=\{m \in M \mid \psi'_{\chi_A}(m) = 1 \} \\
&=(\psi'_{\chi_A})^{-1}(1),
\end{align*}
which is clearly open, since the $\psi'_f$'s are continuous. Therefore, our $\alpha$ is valid and we have shown existence.

The uniqueness part is easy. If a valid $\beta$ commutes with the projections, then:
\begin{align*}
A \in \beta(m) &\ifff \chi_A^{-1}(1) \in \beta(m) \\
&\ifff \psi_{\chi_A}(\beta(m))=1 \\
&\ifff \psi'_{\chi_A}(m)=1 \\
&\ifff A \in \alpha(m),
\end{align*}
so, by extensionality, $\beta = \alpha$.

To show that the codensity monad acts the same as $SC$ on morphisms, we must prove that for every $f: X \to Y \in \Top$ and every $g: Y \to B$ with $B \in \FinSet$, we have that:
$$\psi_g \circ SC(f) = \psi_{g\circ f}$$
But, if $b \in B$ and $U \in SC(X)$:
\begin{align*}
\psi_g(SC(f)(U))=b &\ifff \{y \in Y \mid g(y)=b \} \in SC(f)(U) \\
&\ifff \{x \in X \mid f(x) \in \{y \in Y \mid g(y)=b\} \} \in U \\
&\ifff \{x \in X \mid g(f(x))=b \} \in U \\
&\ifff \psi_{g\circ f}(U)=b.
\end{align*}

Now we turn to the other monadic components. For the unit, as stated in the introduction, we must prove the validity of the equation:
$$\psi_f \circ \eta_X = f,$$
for every map $f$ from $X$ into a finite set $B$. But since $x \in f^{-1}(f(x))$, for every $x \in X$, by the definition of $\eta_X$ we have that:
$$f^{-1}(f(x)) \in \eta_X(x)$$
and so, using the property of $\psi_f$:
$$\psi_f(\eta_X(x)) = f(x).$$

For the multiplication, the corresponding equation is:
$$\psi_f \circ \mu_X = \psi_{\psi_f}$$
But that is equivalent, according to how we defined $\mu_X$, to:
$$\psi_f = \psi_{\psi_f} \circ \eta_{SC(X)},$$
which is a special case of the unit equation and so we have also taken care of the $\mu$ transformation.

The monad laws naturally hold, as it has been pointed out in the Introduction. Hence, the theorem has been proven.
\end{proof}

\begin{cor}
The category of algebras for this codensity monad is equivalent to the category of Stone spaces.
\end{cor}

\begin{proof}
By the theorem, we have that the essential image of this codensity monad is the category of Stone spaces. Since $\mu_X$ was defined as $(\eta_{SC(X)})^{-1}$, we have that $\mu$ is a natural isomorphism and therefore the monad is idempotent. By the argument in the Introduction, its essential image is equivalent to its category of algebras.
\end{proof}

\section{Interlude}\label{sec:inter}

The results just obtained can also be given the following interpretation: maps into finite sets are in a close correspondence with the Boolean algebra of clopen sets of a space, and Stone spaces are exactly the spaces that do not “contain” more information than that already contained in their algebras of clopen sets.

We can ask the natural follow-up question: what spaces are the ones determined by their lattices of open sets? This, and how we can formulate it by using codensity, is the problem that we tackle in the sequel. To begin with, let us remark that the lattice of open sets of a space is of the following special kind.

\begin{defi}
A {\bf frame} is a poset such that finite infima and arbitrary suprema exist and which satisfying the corresponding distributive law:
$$x \wedge (\bigvee_{i \in I} y_i) = \bigvee_{i \in I} (x\wedge y_i)$$
for any $x$ and $\{y_i\}_{i\in I}$ taken from the poset.
\end{defi}

We can organize frames as a category $\catname{Frm}$, taking a frame homomorphism to be a monotone function that preserves these finite infima and arbitrary suprema. We denote by $\Oo$ the (contravariant) functor that maps a topological space to its frame of opens and acts on morphisms by pre-image. (The essential image of $\Oo$ is the category of {\bf topological frames}.)

There is also an analogue of ultrafilters for frames.

\begin{defi}
A {\bf point}, or a {\bf completely prime filter} of a frame $F$ is a subset $\fp$ of $F$ such that:
\begin{enumerate}[(i)]
\item the top element is in $\fp$;
\item the bottom element is {\bf not} in $\fp$;
\item if $A_1 \in \fp$ and $A_2 \in \fp$, then $A_1 \wedge A_2 \in \fp$;
\item if $A_1 \in \fp$ and $A_2 \in F$, then $A_1 \vee A_2 \in \fp$;
\item if $\bigvee\limits_{i \in I} A_i \in \fp$, then there is some $k \in I$ such that $A_k \in \fp$.
\end{enumerate}
\end{defi}

A topology on the set of points of a frame can be given by taking the open sets to be the sets of the form $A^*=\{\fp \mid A \in \fp\}$. This gives another contravariant functor, $\Fp$, from frames to topological spaces (which also acts on morphisms by pre-image). The resulting endofunctor of $\Top$, $\Fp\Oo$, acts on morphisms $f: X \to Y$ by (for an arbitrary $\fp \in \Fp\Oo(X)$):
$$\Fp\Oo(f)(\fp)=\{A \in \Oo(Y) \mid \{x \in X \mid f(x) \in A \} \in \fp \}$$
and is also endowed with a natural “unit” transformation $\eta$, such that for every space $X$, $\eta_X: X \to \Fp\Oo(X)$ is given by:
$$\eta_X(x)=\{A\in\Oo(X)\mid x\in A\}$$
It is easy to check that this defines a well-defined continuous map of topological spaces. We can then simply define {\bf sober spaces} as those spaces $X$ for which this $\eta_X$ is an isomorphism -- in other words, spaces completely “determined” by their open sets. In addition, for every space $X$, $\Fp\Oo(X)$ is sober and is called the {\bf soberification} of $X$. (A standard reference for these, and more, facts on sober spaces is Johnstone \cite{johnst}.) We note that we shall frequently recycle notations introduced in Section~\ref{sec:stone}, like in the case of this $\eta$, if their purpose is similar enough.

What is left to do in order to state our result is to find a subcategory of $\Top$ which plays a role towards sober spaces like $\FinSet$ does towards Stone spaces. From the way Lemma~\ref{lgh} was used in Section~\ref{sec:stone}, we can derive that the use of $\FinSet$ was directly tied to the fact that ultrafilters are characterized by their behavior on finite families of disjoint elements of their Boolean algebra, which are finite families of disjoint opens in the corresponding Stone space (which are necessarily clopens by finiteness and disjointness). Condition (v) in the definition of a completely prime filter suggests that the analogue in this case will be families of not necessarily disjoint opens, indexed by arbitrary sets $I$. In other words, for every set $I$, a space must be found that “represents” $I$-indexed families of opens. The category of these spaces will be the analogue of $\FinSet$.

Take a set $I$ and topologize the set $P(I)$ (the powerset of $I$) as follows: for every finite $F\subseteq I$, define $U_F$ as the set $\{A \subseteq I \mid F \subseteq A\}$. As $U_F \cap U_G = U_{F \cup G}$ and $U_\emptyset=P(I)$, these sets form a basis for a topology on $P(I)$. If we fix a space $X$, then for every $I$-indexed family of opens in X, $\{A_i\}_{i\in I}$, we can define a continuous map $f:X \to P(I)$, where $P(I)$ is considered together with the topology that we just defined, by putting $f(x)=\{i \in I \mid x \in A_i\}$. It can be easily seen that this correspondence between $I$-indexed families of open sets of $X$ and continuous maps to $P(I)$ is well-defined and bijective, hence fulfilling the informal objective mentioned above. (We omit the proof of this bijection, as we shall not use it explicitly.)

A more conceptual way of presenting the model spaces in the last paragraph is as follows. When $I$ is a singleton, let's say $I=\{t\}$, then $P(I)=\{\emptyset,\{t\}\}$, with the above topology has $\{\{t\}\}$ as its single non-trivial open set. This space is called the {\bf Sierpi\'nski space} and we denote it by $\mathbb{S}$. Given that maps from an arbitrary space $X$ to this spaces parametrize, by the discussion above, single open sets in $X$, we have that maps to powers $\mathbb{S}^I$ of the Sierpi\'nski space will parametrize, by the universal property of the product, exactly $I$-indexed families of open sets of $X$. These powers $\mathbb{S}^I$ are therefore, by the Yoneda lemma, homeomorphic to the $P(I)$'s constructed in the last paragraph.

\section{The case of sober spaces}\label{sec:sober}

From now on, we denote by $\mathcal{N}$ the full subcategory of $\Top$ containing exactly the $P(I)$ spaces constructed above, i.e. the arbitrary powers of the  Sierpi\'nski space (which are incidentally sober, but this fact is not needed for the proof below). We can now state the precise analogue of the result for Stone spaces in this new framework:

\begin{teo}
The codensity monad of the inclusion functor $\mathcal{N} \hookrightarrow \Top$ is $(\Fp\Oo,\eta,\mu)$, where, for every $X$, $\mu_X$ is defined as $(\eta_{\Fp\Oo(X)})^{-1}$.
\end{teo}

\begin{proof}
As before, we first need to show:
$$\Fp\Oo(X)=\varprojlim_{I \in \Set,\ f:X\to P(I)} P(I)$$
i.e. to give that space the structure of a universal cone (where $P(I)$ is considered as above, being an object of $\mathcal{N}$). For every $f:X \to P(I)$, we define the map $\psi_f: \Fp\Oo(X) \to P(I)$ by the (extensional) condition:
$$i \in \psi_f(\fp) \ifff f^{-1}(U_{\{i\}})\in\fp$$
for every $i \in I$ and $\fp \in \Fp\Oo(X)$. We check the continuity on the opens which belong to the basis ($F$ is an arbitrary finite subset of $I$):
\begin{align*}
\psi_f^{-1}(U_F)&=\{\fp \mid \psi_f(\fp) \in U_F \} \\
&=\{\fp \mid F \subseteq \psi_f(\fp) \} \\
&=\{\fp \mid \forall i \in F, i \in \psi_f(\fp) \} \\
&=\bigcap_{i \in F}\{\fp \mid f^{-1}(U_{\{i\}})\in\fp \} \\
&=\bigcap_{i \in F} (f^{-1}(U_{\{i\}}))^*,
\end{align*}
which is open. We claim that $(\Fp\Oo(X), \{\psi_f\}_{f:X \to P(I)})$ is a universal cone over that diagram. Let us first check the cone condition, i.e. that for every $f: X \to P(I)$ and every continuous $\phi : P(I) \to P(J)$ we have:
\begin{equation}\label{eq:psi2}
\phi \circ \psi_f = \psi_{\phi \circ f}
\end{equation}
To show that, we first prove the claim that for every $\fp \in \Fp\Oo(X)$ and every open set $D$ of $P(I)$:
\begin{equation}\label{eq:D}
\psi_f(\fp) \in D \ifff f^{-1}(D) \in \fp
\end{equation}
Since the $U_F$'s form a basis, we can write $D$ as $\bigcup_{a \in K} U_{F_a}$ where $K$ is a possibly infinite set and for each $a \in K$, $F_a$ is a finite subset of $I$. Therefore:
\begin{align*}
\psi_f(\fp) \in D &\ifff \exists a \in K, \psi_f(\fp) \in U_{F_a} \\
&\ifff \exists a \in K, F_a \subseteq \psi_f(\fp) \\
&\ifff \exists a \in K, \forall i \in F_a, i \in \psi_f(\fp) \\
&\ifff \exists a \in K, \forall i \in F_a, f^{-1}(U_{\{i\}})\in\fp \\
&\ifff \exists a \in K, \bigcap_{i \in F_a} f^{-1}(U_{\{i\}})\in\fp \\
&\ifff \bigcup_{a \in K}\bigcap_{i \in F_a} f^{-1}(U_{\{i\}})\in\fp \\
&\ifff f^{-1}(\bigcup_{a \in K}\bigcap_{i \in F_a}U_{\{i\}})\in\fp \\
&\ifff f^{-1}(D) \in \fp.
\end{align*}
Now, for every $\fp \in \Fp\Oo(X)$ and every $i \in I$:
\begin{align*}
i \in  \psi_{\phi \circ f}(\fp) &\ifff (\phi \circ f)^{-1}(U_{\{i\}}) \in \fp \\
&\ifff \{x \in X \mid \phi(f(x)) \in U_{\{i\}}\} \in \fp \\
&\ifff \{x \in X \mid f(x) \in \phi^{-1}(U_{\{i\}})\} \in \fp \\
&\ifff \psi_f(\fp) \in \phi^{-1}(U_{\{i\}})&(\text{by}~\ref{eq:D})  \\
&\ifff \phi(\psi_f(\fp)) \in U_{\{i\}} \\
&\ifff i \in (\phi \circ \psi_f)(\fp),
\end{align*}
and so the cone condition is valid.

For showing universality, let $(M, \{\psi'_f\}_{f:X \to P(I)})$ be an arbitrary cone. We must show the existence of a unique $\alpha: M \to \Fp\Oo(X)$ such that for every $f$, we have that:
$$\psi_f \circ \alpha = \psi'_f$$
Fix an arbitrary three-element set $T$ and denote its distinct elements by $t$, $z_1$, $z_2$. We first define a new flavour of characteristic functions, by using the space $P(\{t\}) \in \mathcal{N}$. More precisely, if $X$ is a space and $A \in \Oo(X)$, the function $\chi_A: X \to P(\{t\})$, defined by the rule:
$$\chi_A(x)=\{t\} \ifff x \in A$$
is continuous. We can now define the $\alpha$ on every $m \in M$ by:
$$A \in \alpha(m) \ifff \psi'_{\chi_A}(m)=\{t\},$$
for every $A \in \Oo(X)$.

Let us prove that this map is well defined, i.e. that $\alpha(m)$ is a completely prime filter in $\Oo(X)$ for every $m$. We begin by showing condition (iii) in the definition. Let $A_1, A_2 \in \alpha(m)$. We define the auxiliary maps $\chi: X \to P(\{z_1,z_2\})$, $\chi_1, \chi_2:P(\{z_1,z_2\})\to P(\{t\})$ by:
$$z_i \in \chi(x) \ifff x \in A_i$$
and
$$\chi_i(A)=\{t\} \ifff z_i \in A$$
for every $i \in \{1,2\}, x\in X, A \subseteq \{z_1,z_2\}$. In addition, we define $\chi_3: P(\{z_1,z_2\})\to P(\{t\})$ by:
$$\chi_3(A)=\{t\} \ifff A=\{z_1,z_2\}.$$
These maps are all continuous and satisfy:
$$\chi_i \circ \chi = \chi_{A_i}$$
for every $i \in \{1,2\}$ and
$$\chi_3 \circ \chi = \chi_{A_1 \cap A_2}.$$
By the cone condition for $M$, the following equations also hold:
$$\chi_i \circ \psi'_\chi = \psi'_{\chi_{A_i}}$$
for every $i \in \{1,2\}$ and
$$\chi_3 \circ \psi'_\chi = \psi'_{\chi_{A_1 \cap A_2}}.$$
Since $A_i \in \alpha(m)$, $\psi'_{\chi_{A_i}}(m)=\{t\}$, so $\chi_i(\psi'_\chi(m))=\{t\}$. By the equations above, $z_1, z_2 \in \psi'_\chi(m)$, so $\psi'_\chi(m)=\{z_1,z_2\}$ and $\chi_3(\psi'_\chi(m))=\{t\}$, that is, $\psi'_{\chi_{A_1 \cap A_2}}(m)=\{t\}$. We have, therefore, that $A_1 \cap A_2 \in \alpha(m)$.

For showing condition (iv), we take suitable $A_1$ and $A_2$ and maintain the maps $\chi$ and $\chi_1$ defined above. We define the continuous mapping $\chi_4: P(\{z_1,z_2\})\to P(\{t\})$ by:
$$\chi_4(A)=\{t\} \ifff A\neq\emptyset,$$
so we have that
$$\chi_4 \circ \chi = \chi_{A_1 \cup A_2}$$
and
$$\chi_4 \circ \psi'_\chi = \psi'_{\chi_{A_1 \cup A_2}}.$$
Since $A_1 \in \alpha(m)$, $\psi'_{\chi_{A_1}}(m)=\{t\}$, so $\chi_1(\psi'_\chi(m))=\{t\}$. Then $z_1 \in \psi'_\chi(m)$, so $\psi'_\chi(m)\neq\emptyset$ and $\chi_4(\psi'_\chi(m))=\{t\}$. Thus $\psi'_{\chi_{A_1 \cup A_2}}(m)=\{t\}$ and $A_1 \cup A_2 \in \alpha(m)$.

For showing condition (v), let $\{A_i\}_{i\in I}$ be a family of open sets of $X$ such that $\bigcup_{i \in I} A_i \in \alpha(m)$. We define a continuous map $\rho: X \to P(I)$ by
$$i \in \rho(x) \ifff x \in A_i,$$
a continuous map $\rho': P(I) \to P(\{t\})$ by
$$\rho'(A)=\{t\} \ifff A\neq\emptyset$$
and for each $i \in I$ a continuous map $\rho_i:P(I) \to P(\{t\})$ by
$$\rho_i(A)=\{t\} \ifff i \in A.$$
These maps satisfy (for each $i \in I$):
$$\rho_i \circ \rho = \chi_{A_i}$$
and
$$\rho' \circ \rho = \chi_{\bigcup_{i \in I} A_i}.$$
By using the cone conditions as before, we obtain that $\rho'(\psi'_\rho(m))=\{t\}$, hence $\psi'_\rho(m)$ is nonempty. Then there exists $i \in I$ such that $i \in \psi'_\rho(m)$ and by similar manipulations we obtain that $A_i \in \alpha(m)$.

Condition (ii) can easily be seen as the special case $I = \emptyset$ of condition (v). For condition (i), denote by $\tau : X \to P(\emptyset)$ the unique such continuous mapping ($P(\emptyset)$ being a terminal object in $\Top$) and by $\tau' : P(\emptyset) \to P(\{t\})$ the unique continuous mapping such that $\tau'(\emptyset)=\{t\}$. Then $\tau' \circ \tau = \chi_X$, so $\tau' \circ \psi'_\tau = \psi'_{\chi_X}$. Since only $\{t\}$ is in the image of $\tau'$, we get that $\psi'_{\chi_X}(m)=\{t\}$ and therefore that $X \in \alpha(m)$.

We have now shown that $\alpha(m)$ is a completely prime filter, so $\alpha$ is well defined.

To show that $\alpha$ is continuous, we take an open set $A^*$ of $\Fp\Oo(X)$. Then:
\begin{align*}
\alpha^{-1}(A^*)&=\{m \in M \mid \alpha(m) \in A^* \} \\
&=\{m \in M \mid A \in \alpha(m)\} \\
&=\{m \in M \mid \psi'_{\chi_A}(m)=\{t\}\} \\
&=(\psi'_{\chi_A})^{-1}(\{t\}),
\end{align*}
which is open, since the $\psi'_f$'s are continuous.

To finish up showing the existence of $\alpha$, we must show that the $\alpha$ we constructed commutes with the projections, i.e. that for every valid $f: X \to P(I)$ and every $m \in M$ we have that:
$$\psi_f(\alpha(m))=\psi'_f(m).$$
We first prove that for any open $D$ of $P(I)$:
\begin{equation}\label{eq:D2}
\chi_{f^{-1}(D)} = \chi_D \circ f
\end{equation}
Let $x \in X$ arbitrary. Then:
\begin{align*}
\chi_{f^{-1}(D)}(x)=\{t\}&\ifff x \in f^{-1}(D) \\
&\ifff f(x) \in D\\
&\ifff \chi_D(f(x))= \{t\}\\
&\ifff (\chi_D \circ f)(x)=\{t\}.
\end{align*}
Let now $i \in I$ be arbitrary. Then:
\begin{align*}
i \in \psi_f(\alpha(m))&\ifff f^{-1}(U_{\{i\}}) \in \alpha(m) \\
&\ifff \psi'_{\chi_{f^{-1}(U_{\{i\}})}}(m)=\{t\} \\
&\ifff \psi'_{\chi_{U_{\{i\}}} \circ f}(m)=\{t\}&(\text{by}~\ref{eq:D2}) \\
&\ifff (\chi_{U_{\{i\}}} \circ \psi'_{f})(m)=\{t\} \\
&\ifff \chi_{U_{\{i\}}}(\psi'_{f}(m))=\{t\} \\
&\ifff \psi'_{f}(m) \in U_{\{i\}}\\
&\ifff i \in \psi'_{f}(m).
\end{align*}

To show uniqueness, let $\beta : M \to \Fp\Oo(X)$ be a continuous map that commutes with the projections -- in particular, for every $m \in M$ and every open set $A$ of $X$ we have that:
$$\psi_{\chi_A}(\beta(m))=\psi'_{\chi_A}(m)$$
Then we can deduce:
\begin{align*}
A \in \alpha(m) &\ifff \psi'_{\chi_A}(m)=\{t\} \\
&\ifff \psi_{\chi_A}(\beta(m))=\{t\} \\
&\ifff t \in \psi_{\chi_A}(\beta(m)) \\
&\ifff \chi_A^{-1}(U_{\{t\}}) \in \beta(m) \\
&\ifff \{x \in X \mid \chi_A(x)=\{t\}\} \in \beta(m) \\
&\ifff A \in \beta(m).
\end{align*}

Now we will show that $\Fp\Oo$ acts on morphisms in the same way that a codensity monad would, i.e. that for every $f: X \to Y \in \Top$ and every $g: Y \to P(I)$ we have that:
$$\psi_g \circ \Fp\Oo(f) = \psi_{g\circ f}.$$
For every $i \in I$ and $\fp \in \Fp\Oo(X)$, we can deduce:
\begin{align*}
i \in \psi_g(\Fp\Oo(f)(\fp)) &\ifff g^{-1}(U_{\{i\}}) \in \Fp\Oo(f)(\fp) \\
&\ifff \{x \in X \mid f(x) \in g^{-1}(U_{\{i\}})\} \in \fp \\
&\ifff \{x \in X \mid g(f(x)) \in U_{\{i\}} \} \in \fp \\
&\ifff (g \circ f)^{-1}(U_{\{i\}}) \in \fp \\
&\ifff i \in \psi_{g \circ f}(p).
\end{align*}

To show that the $\eta$ from earlier really is a unit, i.e. that for any $g: X \to P(I)$ we have $\psi_g \circ \eta_X = g$, we do the following deductions for any $x \in X$ and $i \in I$:
\begin{align*}
i \in \psi_g(\eta_X(x)) &\ifff g^{-1}(U_{\{i\}}) \in \eta_X(x) \\
&\ifff x \in g^{-1}(U_{\{i\}}) \\
&\ifff g(x) \in U_{\{i\}} \\
&\ifff i \in g(x).
\end{align*}

To show that the $\mu$ defined earlier is the multiplication component, considering its form, we must really show that for any $g: X \to P(I)$ we have that $\psi_g=\psi_{\psi_g} \circ \eta_{\Fp\Oo(X)}$. But that is a special case of the unit equation.

We are now finished.
\end{proof}

\begin{cor}
The category of algebras for this codensity monad is equivalent to the category of sober spaces.
\end{cor}

\begin{proof}
By the theorem, we have that the essential image of this codensity monad is the category of Stone spaces. Since $\mu_X$ was defined as $(\eta_{\Fp\Oo(X)})^{-1}$, we have that $\mu$ is a natural isomorphism and therefore the monad is idempotent. By the argument in the Introduction, its essential image is equivalent to its category of algebras.
\end{proof}

\section{An alternate proof}\label{sec:alter}

We finally present the announced easier argument one can use to obtain the results above, as suggested by an anonymous reviewer. We can use the fact that that if $h : \mathcal{B} \to \mathcal{A}$ is a codense functor, and $i : \mathcal{A}\to\mathcal{A}_0$ is a reflective subcategory, then the codensity monad for $i \circ h: \mathcal{B}\to\mathcal{A}_0$ is the idempotent monad given by the reflection, or equivalently, the essential image of the codensity monad for $i \circ h$ is the subcategory $\mathcal{A}$. This happens because (taking $F$ to be the left adjoint of $i$) for any object $A$ of $\mathcal{A}_0$:
\begin{align*}
T^{i \circ h}(A) &= \varprojlim_{B \in \mathcal{B},\ f:A\to i(h(B))} i(h(B)) \\
&\simeq i\left(\varprojlim_{B \in \mathcal{B},\ f:A\to i(h(B))} h(B)\right) &\text{(RAPL)}\\
&\simeq i\left(\varprojlim_{B \in \mathcal{B},\ g:F(A)\to h(B)} h(B)\right) &\text{($F$ is left adjoint to $i$)}\\
&\simeq i(T^h(F(A))) &\text{(definition of $T^h$)}\\
&\simeq i(F(A)) &\text{($h$ is codense)}\\
\end{align*}
Here we take $\mathcal{A}_0$ to be $\Top$ and $\mathcal{A}$ to be successively the category of Stone spaces and that of sober spaces. As we mentioned before, those two categories are indeed reflective subcategories of $\Top$. It remains to be proven that the two functors which take finite sets into Stone spaces and $\mathcal{N}$ into sober spaces are codense. In the second case, restricting the well-known duality of sober spaces to topological frames (see \cite[p. 44, Corollary, (1)]{johnst}), the powers of the Sierpi\'nski space correspond to free frames (see \cite[Scholium C1.1.4]{eleph}). Similarly, in the first case, we can see that the finite sets correspond by duality to finite Boolean algebras. So it suffices to prove that the inclusions of finite Boolean algebras into $\BA$ and that of the free frames into topological frames are dense (a notion dual to codensity). For the first one, we use the fact that $\BA$ is a finitarily algebraic category over $\Set$, so the finitely generated free algebras -- and therefore the finite algebras which contain them -- are dense in $\BA$ (a reference for this fact is \cite[Theorem 2.2]{isbell}). Similarly, by \cite[Theorem II.1.2]{johnst}, the category of frames is monadic over $\Set$ -- therefore the free frames form its corresponding Kleisli subcategory. By \cite[Example 4.3]{day}, this subcategory is dense. The proof is finished.

\section{Conclusions and future work}

This paper, in the spirit of Leinster's before it, shows that concepts like Stone spaces or sober spaces are “inevitable” -- that is, they arise naturally by inputting “ordinary” mathematical notions such as finite sets or topological spaces into some standard categorial machinery. It remains to be seen whether the same technique can be applied in other contexts, where conceptual arguments do not necessarily transfer -- for example, categories of ringed or ordered spaces.

\section{Acknowledgements}

I would like to thank the anonymous reviewer for pointing out the alternate proof I outlined earlier.

This work was supported by a grant of the Romanian National Authority for Scientific Research, CNCS - UEFISCDI, project number PN-II-ID-PCE-2011-3-0383.

\end{document}